\documentclass[reqno]{amsart}
\usepackage{amssymb}
\usepackage{amsfonts}
\usepackage{graphicx}

\newtheorem{theorem}{Theorem}
\theoremstyle{plain}

\newtheorem{corollary}{Corollary}

\newtheorem{lemma}{Lemma}

\newtheorem{proposition}{Proposition}
\newtheorem{remark}{Remark}

\newcommand{\func}[1]{\operatorname{#1}}

\numberwithin{equation}{section}

\begin{document}
\title[Secure domination in odd-order Sierpi\'{n}ski graphs]{Secure
domination in odd-order Sierpi\'{n}ski graphs}
\author{Qingcheng Zeng}
\address{School of Mathematical Sciences, Beihang University, Beijing
100083, P. R. China}
\email{qczeng@buaa.edu.cn}
\author{Yumei Xue}
\address{School of Mathematical Sciences, Beihang University, Beijing
100083, P. R. China}
\email{yxue@buaa.edu.cn}
\author{Cheng Zeng}
\address{School of Mathematics and Information Science, Shandong Technology
and Business University, Yantai, Shandong Province 264003, P. R. China}
\email{czeng@sdtbu.edu.cn}
\author{Dehui Wang}
\address{School of Mathematical Sciences, Beihang University, Beijing
100083, P. R. China}
\email{dehuiwang@buaa.edu.cn}

\begin{abstract}
A nonempty vertex set $D$ of a graph is secure if and only if, for every
nonempty subset $X$ of $D$, the closed neighborhood of $X$ contains at least
as many vertices of $D$ as vertices outside $D$. A secure dominating set is
a set that is both secure and dominating, and the secure domination number
is the minimum cardinality of such a set. Previous work determined the
secure domination number of Sierpi\'{n}ski graphs $S_{k}^{n}$ for even $k$
and for odd $k$ when $n=2$. The exact values for odd $k$ and $n\geq 3$
remained open. In this paper, we determine these remaining values.
\end{abstract}

\thanks{Cheng Zeng is the corresponding author. }
\maketitle

\section{Introduction}

All graphs in this paper are finite, undirected, and simple unless
explicitly stated otherwise. We use the standard graph-theoretic terminology
and notation of \cite{Diestel2025}.

For a vertex $v$ of a graph $G=(V,E)$, the \emph{open neighborhood} and 
\emph{closed neighborhood} of $v$ are defined by 
\begin{equation*}
N_{G}(v):=\{u\in V:uv\in E\},\text{\ }N_{G}[v]:=N_{G}(v)\cup \{v\},
\end{equation*}%
respectively. For a vertex subset $X\subset V$, its closed neighborhood is
defined by $N_{G}[X]:=\bigcup_{x\in X}N_{G}[x]$. We omit the graph subscript
when the underlying graph is clear. A set $D\subset V$ is called \emph{%
dominating} if $N[D]=V$.

The notion of security in graphs was introduced by Brigham, Dutton and
Hedetniemi \cite{BDH2007}. Let $S=\{s_{1},\ldots ,s_{t}\}\subset V$ be
nonempty. An \emph{attack} on $S$ is a family of pairwise disjoint vertex
sets $A_{1},\ldots ,A_{t}$ such that $A_{i}\subset N[s_{i}]\setminus S$ for
every $1\leq i\leq t.$ A \emph{defense} of $S$ is a family of pairwise
disjoint vertex sets $B_{1},\ldots ,B_{t}$ such that $B_{i}\subset
N[s_{i}]\cap S$\ for every $1\leq i\leq t.$ An attack is \emph{defendable}
if there exists a defense satisfying $|B_{i}|\geq |A_{i}|$ for $1\leq i\leq
t.$ The set $S$ is called \emph{secure} if every attack on $S$ is
defendable. They further established the following equivalent
characterization \cite{BDH2007}: a nonempty set $S\subset V$ is secure if
and only if 
\begin{equation}
|N[X]\cap S|\geq |N[X]\setminus S|\text{\ for every nonempty }X\subset S.
\label{eq:security}
\end{equation}%
The minimum cardinality of a secure set in $G$ is the \emph{security number} 
$s(G)$. The \emph{secure domination number} $\gamma _{s}(G)$ is the minimum
cardinality of a set that is both secure and dominating, called a \emph{%
secure dominating set}.

The notion of secure domination used here differs from that studied by
Cockayne et al. \cite{CGGMV2005}.

For integers $k,n\geq 1$, write $[k]=\{1,\ldots ,k\}$. The Sierpi\'{n}ski
graph $S_{k}^{n}$ has vertex set $[k]^{n}$. Distinct vertices $u=u_{1}\cdots
u_{n}$ and $v=v_{1}\cdots v_{n}$ are adjacent if and only if there is an
index $h\in \lbrack n]$ such that 
\begin{equation}
\begin{array}{l}
u_{t}=v_{t}\text{ for }t<h, \\ 
u_{h}\neq v_{h}, \\ 
u_{t}=v_{h},\ v_{t}=u_{h}\ \text{for }t>h.%
\end{array}
\label{eq:sier}
\end{equation}%
Sierpi\'{n}ski graphs were introduced by Klav\v{z}ar and Milutinovi\'{c} 
\cite{KM1997}. Their properties have since been studied extensively,
including perfect codes \cite{KMP2002, GKM2005}, colorings and Hamiltonian
cycles \cite{Klavzar2008}, and metric properties \cite{Parisse2009}.
Connectivity has also been investigated for generalized Sierpi\'{n}ski
graphs \cite{KZ2018}. For a comprehensive survey of Sierpi\'{n}ski-type
graphs, see \cite{HKZ2017}. See Fig. \ref{fig1} for an illustration of the
Sierpi\'{n}ski graph $S_{5}^{2}.$ 
\begin{figure}[tbph]
\includegraphics[width=0.6\textwidth]{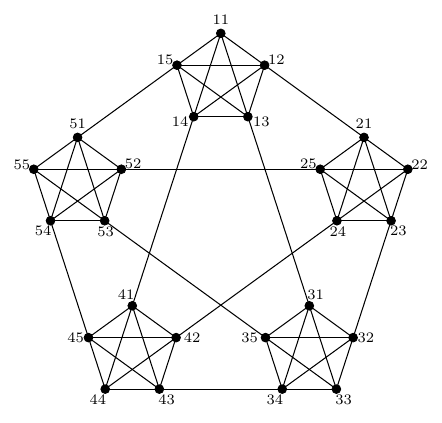}
\caption{Sierpi\'{n}ski graph $S_{5}^{2}$}
\label{fig1}
\end{figure}

Recently, Menon, Chithra and Savitha \cite{MCS2023} determined the security
number of Sierpi\'{n}ski graphs and the secure domination number for even $k$%
, as well as for odd $k$ when $n=2$. Their results are summarized below.

\begin{theorem}[\protect\cite{MCS2023}]
\label{thm:known-menon}For $n,k>1$, $s(S_{k}^{n})=k-1.$ Moreover, $\gamma
_{s}(S_{k}^{n})=\frac{k^{n}}{2}$ if $k$ is even. For odd $k>1$, $\gamma
_{s}(S_{k}^{2})=k\lceil \frac{k}{2}\rceil -1.$ For $k\geq 1,$ we also have $%
\gamma _{s}(S_{k}^{1})=\gamma _{s}(K_{k})=\lceil \frac{k}{2}\rceil ,$ where $%
K_{k}$ denotes the complete graph on $k$ vertices.
\end{theorem}

For odd $k$ and $n>2,$ an upper bound for $\gamma _{s}(S_{k}^{n})$\ was
obtained in \cite{MCS2023}, while its exact value remained open. Our main
result resolves this problem.

\begin{theorem}
\label{thm:main} Let $k\geq 3$ be odd and $n\geq 3$. Then 
\begin{equation*}
\gamma _{s}(S_{k}^{n})=%
\begin{cases}
\dfrac{k^{n}+1}{2}, & k=3\text{\ or\ }k\equiv 1\text{ }(\func{mod}4), \\%
[3mm] 
\dfrac{k^{n}+k^{n-3}}{2}, & k\equiv 3\text{ }(\func{mod}4),\text{\ }k\geq 7.%
\end{cases}%
\end{equation*}
\end{theorem}

\section{Terminal cliques and local inequalities}

The vertices $i^{n}=ii\cdots i$ $(i\in \lbrack k])$\ are called the \emph{%
extreme vertices }of $S_{k}^{n},$ and the remaining vertices are \emph{%
nonextreme}. Each extreme vertex has degree $k-1$, while every nonextreme
vertex has degree $k$.

Fix $n\geq 2$ and write $H:=S_{k}^{n-1}$. For $u=u_{1}\cdots u_{n-1}\in V(H)$
and $a\in \lbrack k]$, let $ua=u_{1}\cdots u_{n-1}a\in V(S_{k}^{n}),$ and
define $Q_{u}:=\{ua:a\in \lbrack k]\}.$ Then 
\begin{equation*}
V(S_{k}^{n})=\bigsqcup_{u\in V(H)}Q_{u},\text{\ }S_{k}^{n}[Q_{u}]\cong K_{k},
\end{equation*}%
where $\bigsqcup $ denotes disjoint union. We call each $Q_{u}$ a \emph{%
terminal clique}. For distinct $u,v\in V(H)$, an edge joins $Q_{u}$ and $%
Q_{v}$ if and only if $uv\in E(H)$, and such an edge\ is unique. We call it
the \emph{bridging edge} corresponding to $uv$. Consequently, contracting
each $Q_{u}$ to a single vertex yields a graph naturally isomorphic to $H$.

If $u=i^{n-1}$ is extreme in $H$, then $i^{n}\in Q_{u}$ is the unique vertex
of $Q_{u}$ having no neighbor outside $Q_{u}$. Every nonextreme vertex of $%
S_{k}^{n}$ has exactly one neighbor outside its terminal clique. These
properties follow directly from the adjacency definition (\ref{eq:sier}).

For the remainder of this section, let $k=2r+1,$ where $r$ is an integer,
and fix an arbitrary vertex set $D\subset V(S_{k}^{n})$. Vertices in $D$ are
called \emph{selected}, and those in $V(S_{k}^{n})\setminus D$ are called 
\emph{unselected}.

A terminal clique $Q_{u}$ is called $D$-\emph{free} if $D\cap
Q_{u}=\varnothing .$ For $u\in V(H),$ define 
\begin{equation*}
s_{u}:=|D\cap Q_{u}|,\text{\ }\mathcal{E}:=\{u\in V(H):s_{u}=0\},
\end{equation*}%
For $u\in V(H)\setminus \mathcal{E}$, define%
\begin{equation*}
e_{u}:=|N_{H}(u)\cap \mathcal{E}|.
\end{equation*}%
Thus $\mathcal{E}$ indexes the $D$-free terminal cliques, and $e_{u}$ counts
the $D$-free terminal cliques adjacent to $Q_{u}$. The notation introduced
in this section will also be used in subsequent sections, relative to the
set $D$ under consideration.

\begin{lemma}
\label{lem:empty}If $D$ is a dominating set, then $\mathcal{E}$ is
independent in $H$ and contains no extreme vertex of $H$. Furthermore, the
following statements hold.

\begin{enumerate}
\item If $v\in \mathcal{E}$ and $uv\in E(H)$, then the endpoint in $Q_{u}$
of the bridging edge corresponding to $uv$ belongs to $D$;

\item $\sum_{u\notin \mathcal{E}}e_{u}=k|\mathcal{E}|.$
\end{enumerate}
\end{lemma}

\begin{proof}
Suppose that $u,v\in \mathcal{E}$ are adjacent in $H$. Since both $Q_{u}$
and $Q_{v}$ are $D$-free, the endpoint in $Q_{u}$ of the bridging edge
corresponding to $uv$ has no neighbor in $D$. This contradicts the
assumption that $D$ is dominating. Hence $\mathcal{E}$ is independent in $H$.

If $i^{n-1}\in \mathcal{E}$ for some $i\in \lbrack k]$, then $i^{n}\in
Q_{i^{n-1}}$ has no neighbor in $D$, since it has no neighbor outside $%
Q_{i^{n-1}}$. This is again a contradiction. Thus $\mathcal{E}$ contains no
extreme vertex of $H$, and consequently $\deg _{H}(v)=k$\ for $v\in \mathcal{%
E}.$

For (1), let $v\in \mathcal{E}$ and $uv\in E(H)$. Since $Q_{v}$ is $D$-free,
the endpoint $q_{v}$ in $Q_{v}$ of the bridging edge $q_{u}q_{v}$
corresponding to $uv$ must be dominated from outside $Q_{v}$. It follows
that $q_{u}\in D$.

Finally, counting the edges between $\mathcal{E}$ and $V(H)\setminus 
\mathcal{E}$ from both sides gives (2).
\end{proof}

\begin{lemma}
\label{lem:minocc}If $D$ is a secure dominating set, then $s_{u}\geq r$ for
every $u\notin \mathcal{E}$. Moreover, if $s_{u}=r$, then $e_{u}=0$, no
extreme vertex of $S_{k}^{n}$ in $Q_{u}$ belongs to $D$, and every vertex of 
$D\cap Q_{u}$ has a neighbor in $D\setminus Q_{u}$.
\end{lemma}

\begin{proof}
Let $u\notin \mathcal{E}$ and $x\in D\cap Q_{u}$. Since $x$ has at most one
neighbor outside $Q_{u}$, the security inequality (\ref{eq:security}) gives%
\begin{equation*}
0\leq |N[x]\cap D|-|N[x]\setminus D|=2s_{u}-k+|(N(x)\setminus Q_{u})\cap
D|-|(N(x)\setminus Q_{u})\setminus D|\leq 2s_{u}-k+1.
\end{equation*}%
Hence $s_{u}\geq (k-1)/2=r$.

Assume $s_{u}=r$. It follows that 
\begin{equation*}
1\leq |(N(x)\setminus Q_{u})\cap D|-|(N(x)\setminus Q_{u})\setminus D|\leq
|(N(x)\setminus Q_{u})\cap D|\leq 1.
\end{equation*}%
Consequently, $|(N(x)\setminus Q_{u})\cap D|=1.$ Hence every vertex of $%
D\cap Q_{u}$ has a neighbor in $D\setminus Q_{u}$. In particular, no extreme
vertex of $S_{k}^{n}$ in $Q_{u}$ belongs to $D$.

Suppose that $e_{u}>0$. Then there exists $v\in \mathcal{E}$ with $uv\in
E(H) $. By Lemma \ref{lem:empty}(1), the endpoint in $Q_{u}$ of the bridging
edge corresponding to $uv$ belongs to $D$. However, its unique outside
neighbor lies in the $D$-free terminal clique $Q_{v}$, which contradicts the
preceding conclusion. Thus $e_{u}=0$.
\end{proof}

\begin{lemma}
\label{lem:minmin}Suppose that $D$ is a secure dominating set. Then, for
every $uv\in E(H)$ with $s_{u}=s_{v}=r$, the bridging edge corresponding to $%
uv$ has at least one endpoint outside $D$.
\end{lemma}

\begin{proof}
Suppose that the bridging edge $xy$ corresponding to $uv$, with $x\in Q_{u}$
and $y\in Q_{v}$, has both endpoints in $D$. It follows that $N[x]=Q_{u}\cup
\{y\},$ $N[y]=Q_{v}\cup \{x\}.$ Hence, for $X=\{x,y\}\subset D$, one has $%
N[X]=Q_{u}\cup Q_{v}.$ Since $Q_{u}\cap Q_{v}=\varnothing $ and $%
s_{u}=s_{v}=r$, we obtain $|N[X]\cap D|=2r$ and $|N[X]\setminus
D|=2(k-r)=2r+2$, which contradicts (\ref{eq:security}).
\end{proof}

For a terminal clique $Q_{u}$ with $s_{u}\geq r+1$, define 
\begin{align*}
b_{u}& =|\{x\in D\cap Q_{u}:x\text{ has a neighbor }x^{\prime }\notin Q_{u}%
\text{ and }x^{\prime }\notin D\}|, \\
c_{u}& =2s_{u}-k-b_{u},
\end{align*}%
and let $t_{u}$ denote the number of vertices $v\in N_{H}(u)$ such that $%
s_{v}=r$ and both endpoints of the bridging edge corresponding to $uv$
belong to $D$.

\begin{lemma}
\label{lem:capacity}Suppose that $D$ is a secure dominating set. If $%
s_{u}\geq r+1$, then $c_{u}\geq t_{u}\geq 0.$
\end{lemma}

\begin{proof}
Fix $u$ with $s_{u}\geq r+1$, and write $t:=t_{u}$. Let $v_{1},\ldots
,v_{t}\in N_{H}(u)$ be the distinct vertices counted by $t_{u}$. For each $%
1\leq i\leq t$, let $x_{i}\in D\cap Q_{v_{i}}$ and $y_{i}\in D\cap Q_{u}$ be
the endpoints of the bridging edge corresponding to $uv_{i}$.

Assume that $b_{u}>0$, and let $B_{u}\subset D\cap Q_{u}$ be the set counted
by $b_{u}$. Then $y_{i}\notin B_{u}$. For each $x\in B_{u},$ its unique
neighbor outside $Q_{u}$ does not belong to $D.$ Moreover, these outside
neighbors are pairwise distinct,\ and lie outside $%
\bigcup_{i=1}^{t}Q_{v_{i}}.$ Set $X:=B_{u}\cup \{x_{1},\ldots ,x_{t}\}.$
Then we have 
\begin{equation*}
N[X]=Q_{u}\sqcup \bigsqcup_{i=1}^{t}Q_{v_{i}}\sqcup \bigsqcup_{x\in
B_{u}}(N(x)\setminus Q_{u}).
\end{equation*}%
The last union consists of exactly the $b_{u}$ unselected outside neighbors
of $B_{u}$.\ Hence (\ref{eq:security}) gives%
\begin{equation*}
s_{u}+tr=|N[X]\cap D|\geq |N[X]\setminus D|=(k-s_{u})+b_{u}+t(k-r).
\end{equation*}%
Since $k=2r+1$, it follows that $c_{u}=2s_{u}-k-b_{u}\geq t.$

Assume that $b_{u}=0$. The case $t=0$ is immediate. If $t>0$, consider $%
\{y_{1},x_{1},\ldots ,x_{t}\}\subset D$. Then $N[\{y_{1},x_{1},\ldots
,x_{t}\}]=Q_{u}\sqcup \bigsqcup_{i=1}^{t}Q_{v_{i}}.$ Thus (\ref{eq:security}%
) gives $s_{u}+tr\geq (k-s_{u})+t(k-r),$ and hence $c_{u}=2s_{u}-k\geq t.$
\end{proof}

For $u\notin \mathcal{E}$, define $x_{u}:=2s_{u}-k-e_{u},$ and write $\Delta
:=2|D|-k^{n}.$

\begin{lemma}
\label{lem:excess}If $D$ is a secure dominating set, then $\Delta
=\sum_{u\notin \mathcal{E}}x_{u}.$ Moreover, for every $u\in V(H)\setminus 
\mathcal{E},$ the following statements hold.

\begin{enumerate}
\item if $s_{u}=r$, then $x_{u}=-1$;

\item if $s_{u}\geq r+1$, then $x_{u}=c_{u}+a_{u}\geq 0,$\ $%
a_{u}:=b_{u}-e_{u}\geq 0;$

\item $x_{u}\equiv 1+e_{u}$ $(\func{mod}2).$
\end{enumerate}
\end{lemma}

\begin{proof}
By Lemma\ \ref{lem:empty}(2), one has%
\begin{equation*}
\sum_{u\notin \mathcal{E}}x_{u}=2\sum_{u\notin \mathcal{E}}s_{u}-k(k^{n-1}-|%
\mathcal{E}|)-\sum_{u\notin \mathcal{E}}e_{u}=2|D|-k(k^{n-1}-|\mathcal{E}%
|)-k|\mathcal{E}|=\Delta .
\end{equation*}%
If $s_{u}=r$, then Lemma \ref{lem:minocc} gives $e_{u}=0$, and hence $%
x_{u}=-1$. Suppose that $s_{u}\geq r+1$. By Lemma \ref{lem:empty}(1), every
edge from $u$ to $\mathcal{E}$ corresponds to a selected endpoint in $Q_{u}$
whose outside neighbor is unselected. It follows that $b_{u}\geq e_{u}$ and $%
a_{u}\geq 0.$ Thus $x_{u}=c_{u}+a_{u}\geq 0$ by Lemma \ref{lem:capacity}.
Since $k$ is odd, $x_{u}\equiv 1+e_{u}$ $(\func{mod}2)$.
\end{proof}

We next give two sufficient conditions for a dominating set to be secure,
which will be used in the subsequent constructions.

\begin{proposition}
\label{prop:suff-high}Let $D$ be a dominating set. Suppose that $s_{u}\geq
r+1$ and $c_{u}\geq 0$ for every $u\in V(H)\setminus \mathcal{E}.$ Then $D$
is secure.
\end{proposition}

\begin{proof}
Let $\varnothing \neq X\subset D$ and $T:=\{u:X\cap Q_{u}\neq \varnothing
\}. $ Since each $Q_{u}$ is complete, we have $\bigcup_{u\in T}Q_{u}\subset
N[X]. $ In particular, 
\begin{equation}
\left\vert N[X]\cap D\right\vert \geq \sum_{u\in T}s_{u}.  \label{eq:cap1}
\end{equation}%
Let $y\in (N[X]\setminus D)\setminus \bigcup_{u\in T}Q_{u}.$ Then $y$ is
adjacent to some $x\in X\cap Q_{u},$ where $u\in T.$\ Since $y\notin
Q_{u}\cup D,$ the vertex $x$ is counted by $b_{u}.$ Consequently, we have $%
\left\vert (N[X]\setminus D)\setminus \bigcup_{u\in T}Q_{u}\right\vert \leq
\sum_{u\in T}b_{u}.$ It follows that 
\begin{equation}
\left\vert N[X]\setminus D\right\vert =\sum_{u\in T}(k-s_{u})+\left\vert
(N[X]\setminus D)\setminus \bigcup_{u\in T}Q_{u}\right\vert \leq \sum_{u\in
T}(k-s_{u}+b_{u}).  \label{eq:minus1}
\end{equation}%
Combining (\ref{eq:cap1}) and (\ref{eq:minus1}), we obtain 
\begin{equation*}
\left\vert N[X]\cap D\right\vert -\left\vert N[X]\setminus D\right\vert \geq
\sum_{u\in T}(2s_{u}-k-b_{u})=\sum_{u\in T}c_{u}\geq 0.
\end{equation*}%
Thus $D$ is secure.
\end{proof}

\begin{proposition}
\label{prop:suff-low} Let $D$ be a dominating set such that $s_{u}\geq r$
for every $u\notin \mathcal{E}$. Suppose that whenever $s_{v}=r$, every
vertex of $D\cap Q_{v}$ has a neighbor in $D\setminus Q_{v}$ belonging to
some terminal clique $Q_{u}$ with $s_{u}\geq r+1$. Suppose also that $%
c_{u}\geq t_{u}$ for every $u$ with $s_{u}\geq r+1$. Then $D$ is secure.
\end{proposition}

\begin{proof}
Let $\varnothing \neq X\subset D$ and $T:=\{u:X\cap Q_{u}\neq \varnothing
\}. $ Then $\bigcup_{u\in T}Q_{u}\subset N[X].$\ Moreover, every $u\in T$
satisfies either $s_{u}=r$ or $s_{u}\geq r+1$.

For each $v\in T$ with $s_{v}=r,$ choose $x\in X\cap Q_{v}.$ By hypothesis, $%
x$ has a selected neighbor $y\in Q_{u}$ outside $Q_{v},$ where $s_{u}\geq
r+1.$ If $u\notin T,$ then $y\in (N[X]\cap D)\setminus \bigcup_{w\in
T}Q_{w}. $ Such vertices $y$ are distinct for distinct $v.$ If $u\in T,$
then $v$ is counted by $t_{u}.$ Counting the two cases, we obtain 
\begin{equation}
\sum_{{v\in T,s_{v}=r}}1\leq \left\vert (N[X]\cap D)\setminus \bigcup_{u\in
T}Q_{u}\right\vert +\sum_{{u\in T,s_{u}\geq r+1}}t_{u}.  \label{eq:cap2}
\end{equation}

On the other hand, let $y\in (N[X]\setminus D)\setminus \bigcup_{u\in
T}Q_{u}.$ Then $y$ is adjacent to some $x\in X\cap Q_{u},$ where $u\in T.$
Since $y\notin Q_{u}\cup D,$ the hypothesis on terminal cliques with $%
s_{u}=r $ implies that $s_{u}\geq r+1.$ Thus $x$ is counted by $b_{u}.$ It
follows that 
\begin{equation}
\left\vert (N[X]\setminus D)\setminus \bigcup_{u\in T}Q_{u}\right\vert \leq
\sum_{u\in T,s_{u}\geq r+1}b_{u}.  \label{eq:minus2}
\end{equation}%
Combining (\ref{eq:cap2}) and (\ref{eq:minus2}), one has 
\begin{eqnarray*}
&&|N[X]\cap D|-|N[X]\setminus D| \\
&=&\sum_{u\in T}(2s_{u}-k)+\left\vert (N[X]\cap D)\setminus \bigcup_{u\in
T}Q_{u}\right\vert -\left\vert (N[X]\setminus D)\setminus \bigcup_{u\in
T}Q_{u}\right\vert \\
&\geq &\sum_{u\in T}(2s_{u}-k)+\sum_{\substack{ v\in T  \\ s_{v}=r}}1-\sum 
_{\substack{ u\in T  \\ s_{u}\geq r+1}}(b_{u}+t_{u}) \\
&=&\sum_{\substack{ u\in T  \\ s_{u}\geq r+1}}(2s_{u}-k-b_{u}-t_{u})+\sum 
_{\substack{ v\in T  \\ s_{v}=r}}(2r-k+1)=\sum_{\substack{ u\in T  \\ %
s_{u}\geq r+1}}(c_{u}-t_{u})\geq 0.
\end{eqnarray*}
\end{proof}

We recall the following result.

\begin{lemma}[\protect\cite{BBJJMM2017}]
\label{lem:half} If $G$ is a graph, then $\gamma _{s}(G)\geq \left\lceil
|V(G)|/2\right\rceil .$
\end{lemma}

\section{The cases $n=3$ and $k=3$}

\subsection{The case $n=3$}

\ 

It is well known that the edges of $K_{2r}$ can be partitioned into $2r-1$
perfect matchings (see, for example, \cite[Theorem 9.1]{Harary1969}). The
following lemma is an elementary consequence.

\begin{lemma}
\label{lem:regular}Let $r\geq 1$ be an integer.

\begin{enumerate}
\item $K_{2r}$ has an $r$-regular spanning subgraph.

\item Let $0$ be a specified vertex of $K_{2r+1}.$\ If $r$ is even, then $%
K_{2r+1}$ has an $r$-regular spanning subgraph. If $r$ is odd, then $%
K_{2r+1} $ has a spanning subgraph in which $0$ has degree $r+1$ and every
other vertex has degree $r$.
\end{enumerate}
\end{lemma}

\begin{proof}
Partition the edges of $K_{2r}$ into $2r-1$ perfect matchings. The union of
any $r$ of them proves (1).

For (2), consider $K_{2r+1}-\{0\}\cong K_{2r}$, and take the union of $r$
pairwise edge-disjoint perfect matchings of this subgraph. Choose $%
\left\lceil r/2\right\rceil $ edges from one of these matchings, and replace
each chosen edge $xy$ by the two edges $0x$ and $0y$. Every vertex other
than $0$ retains degree $r$, while 
\begin{equation*}
\deg (0)=2\left\lceil \frac{r}{2}\right\rceil =\left\{ 
\begin{array}{ll}
r, & r\text{ even}, \\ 
r+1, & r\text{ odd}.%
\end{array}%
\right.
\end{equation*}
\end{proof}

\begin{theorem}
\label{thm:n3}For every odd $k\geq 3$, $\gamma _{s}(S_{k}^{3})=\frac{k^{3}+1%
}{2}.$
\end{theorem}

\begin{proof}
Write $k=2r+1$. By Lemma \ref{lem:half}, one has $\gamma _{s}(S_{k}^{3})\geq 
\frac{k^{3}+1}{2}.$ It remains to construct a secure dominating set
attaining this bound.

Relabel the alphabet $[k]$ by $\mathbb{Z}_{k}$ and let $H:=S_{k}^{2}.$ Then $%
V(H)=\mathbb{Z}_{k}^{2}.$ Set 
\begin{equation*}
\mathcal{E}:=\{i0:i\neq 0\}.
\end{equation*}%
We shall construct $D\subset V(S_{k}^{3})$ such that $Q_{u}$ is $D$-free if
and only if $u\in \mathcal{E}.$ For each $i\neq 0$, one has $%
N_{H}(i0)=\{ij:j\neq 0\}\cup \{0i\}.$ Thus $\mathcal{E}$ is independent and
contains no extreme vertex of $H$. Moreover, we have%
\begin{equation*}
e_{u}=\left\vert N_{H}(u)\cap \mathcal{E}\right\vert =\left\{ 
\begin{array}{ll}
0, & u=00, \\ 
1, & u\in V(H)\setminus (\mathcal{E}\cup \{00\}).%
\end{array}%
\right.
\end{equation*}

Construct $D$ as follows. Since $i0j\sim ij0$\ for $i,j\neq 0,$\ $i00\sim
0ii $\ for $i\neq 0$, we first select 
\begin{equation}
ij0\in D\text{\ for }i,j\neq 0,\text{ }0ii\in D\text{ for }i\neq 0.
\label{eq:select-first-n3}
\end{equation}%
For each $i\neq 0$, the vertices $\{ij:j\neq 0\}$ induce a $K_{2r}$ in $H$.
By Lemma \ref{lem:regular}(1), there exists an $r$-regular spanning subgraph 
$F_{i}$ of this clique. For the clique $\{0j:j\in \mathbb{Z}_{k}\}$, Lemma %
\ref{lem:regular}(2) gives a spanning subgraph $F_{0}$ satisfying $\deg
_{F_{0}}(0j)=r$ for $j\neq 0,$ and 
\begin{equation*}
\deg _{F_{0}}(00)=%
\begin{cases}
r, & r\text{ even,} \\ 
r+1, & r\text{ odd.}%
\end{cases}%
\end{equation*}%
For every edge of $F_{i}$ $(i\in \mathbb{Z}_{k})$, select both endpoints of
the corresponding bridging edge in $S_{k}^{3}$. If $r$ is even, also select $%
000\in Q_{00}$.

These selections are distinct, since different bridging edges have distinct
endpoints, and $000$ has no neighbor outside $Q_{00}$. Consequently, 
\begin{equation*}
\begin{aligned} s_{ij}&=1+\deg_{F_i}(ij)=r+1\text{ } (i,j\ne0),\\
s_{0j}&=1+\deg_{F_0}(0j)=r+1\text{ } (j\ne0), \end{aligned}
\end{equation*}%
while 
\begin{equation*}
s_{00}=%
\begin{cases}
\deg _{F_{0}}(00)+1=r+1, & r\text{ even}, \\ 
\deg _{F_{0}}(00)=r+1, & r\text{ odd}.%
\end{cases}%
\end{equation*}%
Hence $s_{u}=r+1$\ for $u\notin \mathcal{E}.$

By (\ref{eq:select-first-n3}), every vertex of each $D$-free terminal clique 
$Q_{v}$ $(v\in \mathcal{E})$ is dominated. Since $s_{u}=r+1$ for $u\notin 
\mathcal{E}$, all remaining terminal cliques are also dominated. Thus $D$ is
a dominating set. Moreover, the selected vertices in (\ref%
{eq:select-first-n3}) are precisely those having unselected neighbors
outside their terminal cliques. Consequently, $b_{u}=e_{u},$ and hence 
\begin{equation*}
c_{u}=\left\{ 
\begin{array}{ll}
1, & u=00, \\ 
0, & u\in V(H)\setminus (\mathcal{E}\cup \{00\}).%
\end{array}%
\right.
\end{equation*}%
By Proposition \ref{prop:suff-high}, $D$ is secure. Since $|\mathcal{E}|=k-1$%
, one has%
\begin{equation*}
|D|=(k^{2}-|\mathcal{E}|)(r+1)=(k^{2}-k+1)\frac{k+1}{2}=\frac{k^{3}+1}{2}.
\end{equation*}%
Together with the lower bound, this proves the theorem.
\end{proof}

\begin{corollary}
\label{cor:n3-noextreme} If $k\equiv 3$ $(\func{mod}4)$, then $S_{k}^{3}$
admits a minimum secure dominating set containing no extreme vertices.
\end{corollary}

\begin{proof}
Here $r=(k-1)/2$ is odd. The construction in the proof of Theorem \ref%
{thm:n3} then uses only endpoints of bridging edges and no extreme vertex.
\end{proof}

\subsection{The case $k=3$}

\ 

A \emph{Hamiltonian cycle} of a graph $G$ is a cycle containing every vertex
of $G.$

\begin{lemma}[\protect\cite{Klavzar2008}]
\label{lem:ham}For every $m\geq 1$, $S_{3}^{m}$ has a Hamiltonian cycle.
\end{lemma}

\begin{theorem}
\label{thm:k3}For every $n\geq 1$, $\gamma _{s}(S_{3}^{n})=\frac{3^{n}+1}{2}%
. $
\end{theorem}

\begin{proof}
For $n=1$, we have $\gamma _{s}(S_{3}^{1})=\gamma _{s}(K_{3})=2$ by Theorem %
\ref{thm:known-menon}. Let $n\geq 2$, and write $H:=S_{3}^{n-1}$ and $%
q:=3^{n-1}.$ By Lemma \ref{lem:ham}, fix a Hamiltonian cycle $\mathcal{C}$
of $H$. In each of the following constructions, for every retained edge of $%
H $, we select both endpoints of the corresponding bridging edge in $%
S_{3}^{n}$. These selections are distinct. In both cases, the resulting set $%
D$ satisfies $s_{u}\geq 1$ for $u\in V(H),$ so $\mathcal{E}=\varnothing $
and $D $ is dominating.

\smallskip \noindent \textbf{Case 1:} $\boldsymbol{q\equiv 1}$ $\boldsymbol{(%
\func{mod}4)}$\textbf{.} Retain the four edges of a path on five consecutive
vertices of $\mathcal{C}$. Partition the remaining $q-5$ vertices into $%
(q-5)/4$ consecutive $4$-vertex paths, and retain the three edges of each
path.

The values of $s_{u}$ along the $5$-vertex and $4$-vertex paths are,
respectively,%
\begin{equation*}
(1,2,2,2,1)\text{ and }(1,2,2,1).
\end{equation*}%
Every selected vertex has a selected outside neighbor. Hence for each $u$
with $s_{u}=2$, one has $b_{u}=0$ and $c_{u}=1.$ Each terminal clique with $%
s_{v}=1$ has its selected vertex adjacent to a selected vertex in a clique
with $s_{u}=2$, and each such $u$ satisfies $t_{u}\leq 1=c_{u}$. Since $r=1$%
, it follows from Proposition \ref{prop:suff-low} that $D$ is secure. Each
retained edge contributes two distinct selected vertices. Therefore, 
\begin{equation*}
\left\vert D\right\vert =2(4+3\frac{q-5}{4})=\frac{3q+1}{2}=\frac{3^{n}+1}{2}%
.
\end{equation*}

\smallskip \noindent \textbf{Case 2:} $\boldsymbol{q\equiv 3}$ $\boldsymbol{(%
\func{mod}4)}$\textbf{.} Choose an extreme vertex $e$ of $H$, and let $a,b$
be its two neighbors on $\mathcal{C}$. Since $\deg _{H}(e)=2$, $%
N_{H}(e)=\{a,b\}$. Retain the edges $ae$ and $eb$, and additionally select
the unique extreme vertex of $S_{3}^{n}$ contained in $Q_{e}$. Partition the
remaining $q-3$ vertices along $\mathcal{C}-\{a,e,b\}$ into $(q-3)/4$
consecutive $4$-vertex paths, and retain the three edges of each path.

By construction, one has $s_{a}=s_{b}=1$, $s_{e}=3,$ and the values of $%
s_{u} $ along every 4-vertex path are $(1,2,2,1)$. Then $b_{e}=0,$ $c_{e}=3$
and $t_{e}=2.$ For every $u$ with $s_{u}=2$, we again have $b_{u}=0$ and $%
c_{u}=1\geq t_{u}.$ Moreover, every terminal clique with $s_{v}=1$ has its
selected vertex adjacent to a selected vertex in a clique with $s_{u}\geq 2$%
. Proposition \ref{prop:suff-low} therefore implies that $D$ is secure. The
two retained edges incident with $e$, together with the additionally
selected extreme vertex, contribute five selected vertices. Hence%
\begin{equation*}
\left\vert D\right\vert =5+6\frac{q-3}{4}=\frac{3q+1}{2}=\frac{3^{n}+1}{2}.
\end{equation*}

In both cases, the constructed secure dominating set attains the lower bound
from Lemma \ref{lem:half}. Therefore, we have $\gamma _{s}(S_{3}^{n})=\frac{%
3^{n}+1}{2}.$
\end{proof}

\section{The cases of odd $k\geq 5$}

We now consider the remaining cases $k\geq 5$, distinguishing the two
residue classes of $k$ modulo $4$.

\subsection{The case $k\equiv 1$ $(\func{mod}4)$}

\ 

Let $\func{dist}_{G}(x,y)$ denote the distance between vertices $x$ and $y$
in a connected graph $G.$ A set $C\subset V(G)$ is a \emph{1-code} if $\func{%
dist}_{G}(x,y)\geq 3$ for all distinct $x,y\in C$. A $1$-code $C$ is a \emph{%
1-perfect code} if $V(G)=\bigsqcup_{c\in C}N_{G}[c].$

We use the following lemma on codes in Sierpi\'{n}ski graphs.

\begin{lemma}[\protect\cite{GKM2005}]
\label{lem:codes}Let $m\geq 1$. For each $i\in \lbrack k]$, the following
statements hold.

\begin{enumerate}
\item If $m$ is odd, then $S_{k}^{m}$ admits a $1$-perfect code $P_{i}^{m}$
whose only extreme vertex is $i^{m}$.

\item If $m$ is even, then $S_{k}^{m}$ admits a $1$-code $B_{i}^{m}$
containing no extreme vertex such that $V(S_{k}^{m})\setminus \{{i^{m}\}}%
=\bigsqcup_{v\in B_{i}^{m}}N_{S_{k}^{m}}[v].$
\end{enumerate}
\end{lemma}

\begin{corollary}
\label{cor:code-size} Let $m\geq 1,$ $q=k^{m},$ and $i\in \lbrack k]$.

\begin{enumerate}
\item If $m$ is even, then $|B_{i}^{m}|=\frac{q-1}{k+1}.$

\item If $m$ is odd, then $|P_{i}^{m}|=\frac{q+1}{k+1}.$
\end{enumerate}
\end{corollary}

\begin{proof}
If $m$ is even, then every vertex of $B_{i}^{m}$ is nonextreme and hence has
a closed neighborhood of size $k+1$. Since these closed neighborhoods
partition $V(S_{k}^{m})\setminus {i^{m}}$, we obtain $(k+1)\left\vert
B_{i}^{m}\right\vert =q-1,$ which proves (1).

If $m$ is odd, then $P_{i}^{m}$ contains exactly one extreme vertex, whose
closed neighborhood has size $k$, while every other code vertex has a closed
neighborhood of size $k+1$. Since these closed neighborhoods partition $%
V(S_{k}^{m})$, it follows that $q=k+\left( \left\vert P_{i}^{m}\right\vert
-1\right) (k+1)=(k+1)\left\vert P_{i}^{m}\right\vert -1,$ which proves (2).
\end{proof}

A \emph{multigraph} may contain parallel edges and is \emph{loopless} if it
has no loops. A $2$-\emph{factor} is a $2$-regular spanning subgraph.

\begin{lemma}[\protect\cite{Petersen1891}]
\label{lem:petersen}Every loopless $2d$-regular multigraph admits a
decomposition into $d$ pairwise edge-disjoint $2$-factors.
\end{lemma}

\begin{theorem}
\label{thm:k1} If $k\equiv 1$ $(\func{mod}4)$, $k\geq 5$ and $n\geq 3$, then 
$\gamma _{s}(S_{k}^{n})=\frac{k^{n}+1}{2}.$
\end{theorem}

\begin{proof}
Write $k=2r+1$, where $r$ is even. Let $m:=n-1,$ $H:=S_{k}^{m},$ $q:=k^{m}.$
Lemma \ref{lem:half} gives the lower bound. We construct a secure dominating
set $D\subset V(S_{k}^{n})$ of the same cardinality.

In both cases below, we prescribe a set $\mathcal{E}\subset V(H)$ indexing
the $D$-free terminal cliques and choose a spanning subgraph $F$ of $J:=H-E$%
. The construction of $D$ uses the following rules:

\begin{enumerate}
\item[(i)] For every edge $uv\in E(H)$ with $u\in V(J)$ and $v\in \mathcal{E}
$, select the endpoint in $Q_{u}$ of the corresponding bridging edge.

\item[(ii)] For every edge $uv\in E(F)$, select both endpoints of the
corresponding bridging edge.
\end{enumerate}
Additional selections will be specified in each case.

\smallskip \noindent \textbf{Case 1: }$\boldsymbol{m}$\textbf{\ is even.}
Fix $i\in \lbrack k]$ and let 
\begin{equation*}
\mathcal{E}:=B_{i}^{m},\text{ }a:=i^{m},\text{ }J:=H-\mathcal{E}.
\end{equation*}%
By Lemma \ref{lem:codes}, the set $\mathcal{E}$ is independent and contains
no extreme vertex of $H.$ Moreover, one has $e_{a}=0,$ and $e_{u}=1$ for $%
u\in V(J)\setminus \{a\}.$ It follows that%
\begin{equation*}
\deg _{J}(u)=\left\{ 
\begin{array}{ll}
2r-1, & u\neq a\text{ is extreme in }H, \\ 
2r, & \text{otherwise.}%
\end{array}%
\right.
\end{equation*}%
Add a new vertex $\infty $ and one auxiliary edge joining $\infty $ to each
of the $k-1$ extreme vertices other than $a.$ The resulting graph $J^{+}$ is 
$2r$-regular. By Lemma \ref{lem:petersen}, its edge set decomposes into $r$
pairwise edge-disjoint 2-factors. Since $r$ is even, the union of any $r/2$
of them is an $r$-regular spanning subgraph $F^{+}$ of $J^{+}$. Let $%
F:=F^{+}-\{\infty \}.$

Construct $D$ by rules (i) and (ii). In addition, select the extreme vertex
of $S_{k}^{n}$ in $Q_{a}.$ For every other extreme vertex $u$ of $H$ with $%
u\infty \in E(F^{+})$, also select the extreme vertex of $S_{k}^{n}$ in $%
Q_{u}.$

For each $u\in V(J)\setminus \{a\}$, the vertices selected through $F$,
together with the additional extreme vertex when $u\infty \in E(F^{+})$,
contribute exactly $r$ vertices to $Q_{u}$. Together with the endpoint
selected by (i), this gives $s_{u}=r+1$ for $u\in V(J)\setminus \{a\}.$
Similarly, $s_{a}=\deg _{F}(a)+1=r+1.$ Thus every terminal clique indexed by 
$J$ contains exactly $r+1$ selected vertices, while those indexed by $%
\mathcal{E}$\ are $D$-free.

Since $\mathcal{E}$ contains no extreme vertex of $H$, every vertex of a $D$%
-free terminal clique is dominated through its outside neighbor selected by
(i). The remaining terminal cliques are dominated since $s_{u}=r+1$. Hence $%
D $ is dominating.

By construction, one has $b_{u}=e_{u},$ and 
\begin{equation*}
c_{u}=\left\{ 
\begin{array}{ll}
1, & u=a, \\ 
0, & u\in V(J)\setminus \{a\}.%
\end{array}%
\right.
\end{equation*}%
Thus Proposition \ref{prop:suff-high} implies that $D$ is secure. By
Corollary \ref{cor:code-size}, we have 
\begin{equation*}
\left\vert D\right\vert =(q-\left\vert \mathcal{E}\right\vert )(r+1)=(q-%
\frac{q-1}{k+1})\frac{k+1}{2}=\frac{kq+1}{2}=\frac{k^{n}+1}{2}.
\end{equation*}

\smallskip \noindent \textbf{Case 2: }$\boldsymbol{m}$\textbf{\ is odd.}
Here $m\geq 3$. Choose distinct $i,j\in \lbrack k]$ and let $P_{i}^{m}$ be
the $1$-perfect code from Lemma \ref{lem:codes}. Let $a:=i^{m},$\ $%
x:=i^{m-1}j$ and $y:=i^{m-2}ji.$ By (\ref{eq:sier}), $N_{H}[a]=\{i^{m-1}t:t%
\in \lbrack k]\}$ and $N_{H}[x]=N_{H}[a]\cup \{y\}.$ Then $a$ and $u$ belong
to same terminal clique of $H,$ and $y$ is the unique neighbor of $x$
outside this clique. Define 
\begin{equation*}
\mathcal{E}:=(P_{i}^{m}\setminus \{a\})\cup \{x\},\text{ }J:=H-\mathcal{E}.
\end{equation*}%
Since $P_{i}^{m}$ is a 1-code containing $a$, neither $x$ or $y$ belongs to $%
P_{i}^{m},$ and $x$ has no neighbor in $P_{i}^{m}\setminus \{a\}$.\ Thus $%
\mathcal{E}$ is independent and contains no extreme vertex of $H$. Since $%
y\notin P_{i}^{m}\cup N_{H}[a],$ the vertex $y$ has exactly one neighbor in $%
P_{i}^{m}\setminus \{a\}$.\ Replacing $a$ with $x$ preserves the coverage of 
$N_{H}[a]$ and additionally covers $y.$ Hence we have $e_{y}=2,$ and $%
e_{u}=1 $ for $u\in V(J)\setminus \{y\}.$ It follows that for every $u\in
V(J),$ 
\begin{equation*}
\deg _{J}(u)=\deg _{H}(u)-e_{u}=\left\{ 
\begin{array}{ll}
2r-1, & u=y\text{ or }u\text{ is extreme in }H, \\ 
2r, & \text{otherwise.}%
\end{array}%
\right.
\end{equation*}

There are $k+1=2r+2$ vertices of degree $2r-1$. Pair them arbitrarily and
add one auxiliary edge for each pair, allowing parallel edges. The resulting
loopless multigraph $J^{+}$ is $2r$-regular. By Lemma \ref{lem:petersen},
the union of $r/2$ of its pairwise edge-disjoint $2$-factors is an $r$%
-regular spanning subgraph $F^{+}$. Let $F$ be obtained by deleting the
auxiliary edges from $F^{+}$.

Construct $D$ by rules (i) and (ii). For every extreme vertex $u$ of $H$
whose auxiliary edge belongs to $F^{+}$, also select the extreme vertex of $%
S_{k}^{n}$ in $Q_{u}.$ If the auxiliary edge incident with $y$ belongs to $%
F^{+}$, then $\deg _{F}(y)=r-1$. Since $\deg _{J}(y)-\deg _{F}(y)=r>0$, we
may choose an edge $yz\in E(J)\setminus E(F)$ and select only the endpoint
in $Q_{y}$ of the bridging edge corresponding to $yz.$

By construction, we have $s_{y}=r+2$, and $s_{u}=r+1$ for $u\in
V(J)\setminus \{y\}.$ As in Case 1, rule (i) dominates every $D$-free
terminal clique. Thus $D$ is dominating.

For $u\in V(J)\setminus \{y\}$, only the endpoint selected by (i) has an
unselected outside neighbor. Therefore, $b_{u}=e_{u}=1$ and $c_{u}=0.$ For $%
u=y$, the two endpoints selected by (i) have unselected outside neighbors.
The additional one-sided selection, when made, contributes one further such
vertex. Consequently,%
\begin{equation*}
(b_{y},c_{y})=\left\{ 
\begin{array}{ll}
(2,1), & \text{if the auxiliary edge at }y\text{ is not in }F^{+}, \\ 
(3,0), & \text{if the auxiliary edge at }y\text{ is in }F^{+}.%
\end{array}%
\right.
\end{equation*}%
Hence $c_{u}\geq 0$ for every $u\in V(J),$ and Proposition \ref%
{prop:suff-high} implies that $D$ is secure. By Corollary \ref{cor:code-size}%
, one has%
\begin{equation*}
\left\vert D\right\vert =(q-\left\vert \mathcal{E}\right\vert )(r+1)+1=(q-%
\frac{q+1}{k+1})\frac{k+1}{2}+1=\frac{kq+1}{2}=\frac{k^{n}+1}{2}.
\end{equation*}

In both cases, the constructed secure dominating set attains the lower bound
from Lemma \ref{lem:half}. Therefore, $\gamma _{s}(S_{k}^{n})=\frac{k^{n}+1}{%
2}.$
\end{proof}

\subsection{The case $k\equiv 3$ $(\func{mod}4)$ and $k\geq 7$}

\ 

Write $k=2r+1\equiv 3$ $(\func{mod}4),$ where $r\geq 3$ is odd, and fix $%
n\geq 3$. To establish the lower bound, let $D\subset V(S_{k}^{n})$ be an
arbitrary secure dominating set. Set $H:=S_{k}^{n-1}$ and retain the
notation of Section $2$ relative to $D$. Let $K:=S_{k}^{n-2},$\ $%
q:=|V(K)|=k^{n-2}.$ Write $C_{z}:=\{za:a\in \lbrack k]\}$ for $z\in V(K).$
Since $H[C_{z}]\cong K_{k}$ and $\mathcal{E}$\ is independent in $H$ by
Lemma \ref{lem:empty}, one has $\left\vert C_{z}\cap \mathcal{E}\right\vert
\leq 1$ for $z\in V(K).$ Define 
\begin{equation*}
Z:=\{z:C_{z}\cap \mathcal{E}\neq \varnothing \},\text{\ }U:=V(K)\setminus Z.
\end{equation*}

For each $z\in Z$, the unique vertex of $C_{z}\cap \mathcal{E}$ is
nonextreme in $H$ by Lemma \ref{lem:empty} and hence has a unique neighbor
outside $C_{z}$. If the neighbor belongs to $C_{w}$, draw an arc $%
z\rightarrow w$. Thus each $z\in Z$ has exactly one outgoing arc, and there
are no directed $2$-cycles. Let $B$ be the number of arcs from $Z$ to $Z$.

For $z\in U$, define $d_{z}:=|\{v\in C_{z}:e_{v}=0\}|,$\ $d:=\sum_{z\in
U}d_{z}.$

\begin{lemma}
\label{lem:direction-count}With the notation above, $B-d=q-(k+1)|U|.$
\end{lemma}

\begin{proof}
If $z\in U$, we have $C_{z}\cap \mathcal{E}=\varnothing .$ Since every
vertex of $C_{z}$ has at most one neighbor outside $C_{z}$, the vertices $%
v\in C_{z}$ with $e_{v}=1$\ correspond bijectively to the arcs from $Z$ to $%
z.$\ All other vertices of $C_{z}$ have $e_{v}=0$.

Since every $z\in Z$ has exactly one outgoing arc, there are precisely $%
|Z|-B $ arcs from $Z$ to $U$. Consequently, 
\begin{equation*}
d=k\left\vert U\right\vert -(|Z|-B)=k\left\vert U\right\vert
-(q-|U|-B)=B+(k+1)\left\vert U\right\vert -q.
\end{equation*}%
Thus $B-d=q-(k+1)|U|.$
\end{proof}

For $z\in V(K)$, define $\delta (C_{z}):=\sum_{v\in C_{z}\setminus \mathcal{E%
}}x_{v}.$

\begin{lemma}
\label{lem:U-excess}For every $z\in U$, $\delta (C_{z})\geq 1.$ Moreover, if 
$d_{z}=0$, then $\delta (C_{z})\geq 2.$
\end{lemma}

\begin{proof}
Let $z\in U$. Since $C_{z}\cap \mathcal{E=}\varnothing $, we have $e_{v}\in
\{0,1\}$ for every $v\in C_{z}.$ Let $\ell :=|\{v\in C_{z}:s_{v}=r\}|.$

Assume that $\ell >0$. By Lemma \ref{lem:minocc}, each $Q_{v}$ with $s_{v}=r$
has $r$ selected vertices whose outside neighbors also belong to $D.$ These
correspond to $r$ distinct neighbors of $v$ in $H$, all satisfying $%
s_{w}\geq r+1$ by Lemmas \ref{lem:minocc} and \ref{lem:minmin}. Since $v$
has at most one neighbor outside $C_{z}$, at least $r-1$ of these neighbors
lie in $C_{z}$. Hence 
\begin{equation*}
\sum_{w\in C_{z},s_{w}\geq r+1}t_{w}\geq (r-1)\ell .
\end{equation*}%
By Lemmas \ref{lem:capacity} and \ref{lem:excess}, we have 
\begin{eqnarray*}
\delta (C_{z}) &=&-\ell +\sum_{w\in C_{z},s_{w}\geq r+1}x_{w}\geq -\ell
+\sum_{w\in C_{z},s_{w}\geq r+1}c_{w} \\
&\geq &-\ell +\sum_{w\in C_{z},s_{w}\geq r+1}t_{w}\geq (r-2)\ell \geq 1.
\end{eqnarray*}

Now suppose that $\ell =0$ and $d_{z}>0.$ By Lemma \ref{lem:excess}(2), we
have $x_{u}\geq 0$ for every $u\in C_{z}.$ Since $d_{z}>0,$ there exists $%
v\in C_{z}$ with $e_{v}=0.$ Thus $x_{v}=2s_{v}-k\equiv 1$ $(\func{mod}2),$
so $x_{v}\geq 1.$ Consequently, $\delta (C_{z})\geq 1.$

Finally, suppose that $\ell =0$ and $d_{z}=0.$ Then $e_{v}=1$ and $s_{v}\geq
r+1$ for every $v\in C_{z}.$ Hence $x_{v}=2(s_{v}-r-1)\geq 0,$ so every $%
x_{v}$ is a nonnegative even integer. Assume that $\delta (C_{z})=0.$ Then $%
x_{v}=0$ for every $v\in C_{z},$ which implies $s_{v}=r+1.$ Lemmas \ref%
{lem:capacity} and \ref{lem:excess}(2) imply $c_{v}=a_{v}=0,$ $%
b_{v}=e_{v}=1. $ It follows that 
\begin{equation*}
1=e_{v}=\left\vert \left( N_{H}(v)\setminus C_{z}\right) \cap \mathcal{E}%
\right\vert \leq \left\vert \left( N_{H}(v)\setminus C_{z}\right)
\right\vert \leq 1.
\end{equation*}%
Thus every $v\in C_{z}$ has a unique neighbor outside $C_{z}$ in $H,$ and
this neighbor belongs to $\mathcal{E}$. By Lemma \ref{lem:empty}(1), the
corresponding bridging edge has a selected endpoint in $Q_{v}$ and an
unselected endpoint in the $D$-free terminal clique. Since $s_{v}=r+1,$ $%
b_{v}=1,$ and $v$ is nonextreme in $H$, it follows that the remaining $r$
selected vertices of $Q_{v}$ have selected outside neighbors. These
correspond to $r$ distinct edges of $H[C_{z}]$. Consequently, these edges
form an $r$-regular spanning subgraph of $H[C_{z}]\cong K_{k}.$ This
contradicts the handshaking lemma, since $kr$ is odd. Thus $\delta (C_{z})>0$%
. Since all $x_{v}$ are even, we have $\delta (C_{z})\geq 2$.
\end{proof}

\begin{lemma}
\label{lem:Z-excess}For every $z\in Z$, $\delta (C_{z})\geq \func{indeg}%
_{Z}(z),$ where $\func{indeg}_{Z}(z)$ denotes the number of arcs from $Z$ to 
$z.$ Consequently, $\sum_{z\in Z}\delta (C_{z})\geq B.$
\end{lemma}

\begin{proof}
Fix $z\in Z.$ Since $C_{z}\cap \mathcal{E}$ consists of exactly one vertex
and $C_{z}$ is complete, we have $e_{v}\geq 1$ for $v\in C_{z}\setminus 
\mathcal{E}$. By Lemma \ref{lem:minocc}, $s_{v}\geq r+1$, and hence $%
x_{v}\geq 0$ by Lemma \ref{lem:excess}(2).

For each arc $w\rightarrow z$ with $w\in Z,$ its corresponding endpoint $%
v\in C_{z}$ is adjacent to the unique vertices of $C_{w}\cap \mathcal{E}$
and $C_{z}\cap \mathcal{E}$. Since $\mathcal{E}$ is independent, $v\notin 
\mathcal{E}.$ Then $e_{v}=2,$ and distinct incoming arcs have distinct
endpoints. By Lemma \ref{lem:excess}(3), one has $x_{v}\equiv 1+e_{v}\equiv
1 $ $(\func{mod}2).$ Since $x_{v}\geq 0,$ we have $x_{v}\geq 1$ for every
such endpoint. Therefore, $\delta (C_{z})\geq \func{indeg}_{Z}(z).$ Summing
over $z\in Z,$ we obtain $\sum_{z\in Z}\delta (C_{z})\geq B.$
\end{proof}

\begin{theorem}
\label{thm:k3mod4-lower}Let $k\equiv 3$ $(\func{mod}4)$, $k\geq 7$, and $%
n\geq 3$. Then every secure dominating set $D$ of $S_{k}^{n}$ satisfies $%
|D|\geq \frac{k^{n}+k^{n-3}}{2}.$
\end{theorem}

\begin{proof}
Since the cliques $C_{z}$ partition $V(H),$ Lemmas \ref{lem:excess}, \ref%
{lem:U-excess}, and \ref{lem:Z-excess} give%
\begin{equation}
\Delta =\sum_{z\in V(K)}\delta (C_{z})\geq \left\vert U\right\vert
+\left\vert \{z\in U:d_{z}=0\}\right\vert +B  \label{eq:Delta}
\end{equation}%
Since $d_{z}\geq 1$ whenever $d_{z}>0,$ we have 
\begin{equation*}
d=\sum_{z\in U}d_{z}\geq \left\vert \{z\in U:d_{z}>0\}\right\vert
=\left\vert U\right\vert -\left\vert \{z\in U:d_{z}=0\}\right\vert .
\end{equation*}%
Consequently, by (\ref{eq:Delta}) and Lemma \ref{lem:direction-count}, one
has $\Delta \geq 2\left\vert U\right\vert +B-d=q-(k-1)\left\vert
U\right\vert .$ On the other hand, (\ref{eq:Delta}) also gives $\Delta \geq
\left\vert U\right\vert .$\ Combining these two bounds, we obtain $k\Delta
=(k-1)\Delta +\Delta \geq (k-1)\left\vert U\right\vert +q-(k-1)\left\vert
U\right\vert =q.$ Hence $2\left\vert D\right\vert -k^{n}=\Delta \geq \frac{q%
}{k},$ which implies $|D|\geq \frac{k^{n}+k^{n-3}}{2}.$
\end{proof}

\begin{lemma}
\label{lem:blocks}Let $n\geq 3.$ Fixing the first $n-3$ coordinates of the
vertices of $S_{k}^{n}$ gives an induced subgraph isomorphic to $S_{k}^{3}$.
Moreover, every edge joining two distinct such subgraphs has both endpoints
extreme in their respective copies of $S_{k}^{3}.$
\end{lemma}

\begin{proof}
By the adjacency definition (\ref{eq:sier}), vertices with the same first $%
n-3$ coordinates are adjacent precisely when their last three coordinates
are adjacent in $S_{k}^{3}.$ Hence each fixed prefix induces a copy of $%
S_{k}^{3}.$

Let $u=u_{1}\cdots u_{n}$ and $v=v_{1}\cdots v_{n}$ be adjacent vertices in
distinct copies of $S_{k}^{3}$, and let $h\leq n-3$ be their first differing
coordinate. By (\ref{eq:sier}), we have $u_{t}=v_{h},\ v_{t}=u_{h}\ $for $%
t>h.$ In particular, $u_{n-2}=u_{n-1}=u_{n}=v_{h}$ and $%
v_{n-2}=v_{n-1}=v_{n}=u_{h}.$ Thus both endpoints are extreme vertices of
their respective copies of $S_{k}^{3}.$
\end{proof}

\begin{theorem}
\label{thm:k3mod4-upper}If $k\equiv 3$ $(\func{mod}4)$, $k\geq 7$, and $%
n\geq 3$, then $\gamma _{s}(S_{k}^{n})\leq \frac{k^{n}+k^{n-3}}{2}.$
\end{theorem}

\begin{proof}
For $n=3$, the result follows from Theorem \ref{thm:n3}. Let $n>3$. By Lemma %
\ref{lem:blocks}, fixing the first $n-3$ coordinates partitions $%
V(S_{k}^{n}) $ into the vertex sets of $k^{n-3}$ induced subgraphs $%
G_{1},\ldots ,G_{k^{n-3}},$ each of which is isomorphic to $S_{k}^{3}.$

By Theorem \ref{thm:n3} and Corollary \ref{cor:n3-noextreme}, each $G_{i}$
admits a secure dominating set $D_{i}$ such that $\left\vert
D_{i}\right\vert =\frac{k^{3}+1}{2},$ and $D_{i}$ contains no extreme vertex
of $G_{i}.$ Let $D:=\bigcup_{i=1}^{k^{n-3}}D_{i}.$ By Lemma \ref{lem:blocks}%
, no vertex of $D$ is incident with an edge between distinct copies.
Consequently, for every nonempty $X\subset D,$ we have $N[X]=%
\bigsqcup_{i=1}^{k^{n-3}}N_{G_{i}}[X\cap D_{i}].$ Since each $D_{i}$ is
secure in $G_{i},$ we obtain 
\begin{equation*}
|N[X]\cap D|-|N[X]\setminus D|=\sum_{i=1}^{k^{n-3}}\left( \left\vert
N_{G_{i}}[X\cap D_{i}]\cap D_{i}\right\vert -\left\vert N_{G_{i}}[X\cap
D_{i}]\setminus D_{i}\right\vert \right) \geq 0.
\end{equation*}%
Thus $D$ is secure in $S_{k}^{n}.$ Moreover, since each $D_{i}$ dominates $%
G_{i},$ it follows that 
\begin{equation*}
N[D]\supset
\bigcup_{i=1}^{k^{n-3}}N_{G_{i}}[D_{i}]=%
\bigcup_{i=1}^{k^{n-3}}V(G_{i})=V(S_{k}^{n}).
\end{equation*}%
Hence $D$ is dominating. Finally, we have $\left\vert D\right\vert
=\sum_{i=1}^{k^{n-3}}\left\vert D_{i}\right\vert =\frac{k^{n}+k^{n-3}}{2}.$
This proves the asserted upper bound.
\end{proof}

\begin{theorem}
\label{thm:k3mod4} If $k\equiv 3$ $(\func{mod}4)$, $k\geq 7$, and $n\geq 3$,
then $\gamma _{s}(S_{k}^{n})=\frac{k^{n}+k^{n-3}}{2}.$
\end{theorem}

\begin{proof}
The result follows immediately from Theorems \ref{thm:k3mod4-lower} and \ref%
{thm:k3mod4-upper}.
\end{proof}

Theorem \ref{thm:main} follows from Theorems \ref{thm:k3}, \ref{thm:k1}, and %
\ref{thm:k3mod4}. Together with Theorem \ref{thm:known-menon}, we obtain the
complete formula for odd $k.$

\begin{corollary}
\label{cor:complete} Let $k\geq 3$ be odd and $n\geq 1$. Then 
\begin{equation*}
\gamma _{s}(S_{k}^{n})=%
\begin{cases}
\dfrac{k+1}{2}, & n=1, \\[2mm] 
\dfrac{k(k+1)}{2}-1, & n=2, \\[2mm] 
\dfrac{k^{n}+1}{2}, & n\geq 3,\text{ }k=3\text{ or }k\equiv 1\text{ }(\func{%
mod}4), \\[2mm] 
\dfrac{k^{n}+k^{n-3}}{2}, & n\geq 3,\text{ }k\equiv 3\text{ }(\func{mod}4),%
\text{ }k\geq 7.%
\end{cases}%
\end{equation*}
\end{corollary}

\begin{remark}
For $k\equiv 3$ $(\func{mod}4)$ and $k\geq 7$, the correction term $%
k^{n-3}/2 $ reflects the parity obstruction in Lemma \ref{lem:U-excess},
while the construction in Theorem \ref{thm:k3mod4-upper} attains the
corresponding lower bound.
\end{remark}

\section*{Acknowledgments}

The research is partially supported by the National Natural Science
Foundation of China (No. 12301107) and the Natural Science Foundation of
Shandong Province, China (No. ZR202209010046).

\section*{Conflict of interest}

The authors declare that they have no financial or non-financial interests
relevant to the content of this article.

\end{document}